\documentclass{amsart}
\usepackage{amsmath}
\usepackage{amsfonts}
\usepackage{amssymb}
\usepackage{graphicx}%
\newtheorem{mydef}{Definition}[section]
\newtheorem{myrem}{Remark}[section]
\newtheorem{mytheo}{Theorem}[section]
\newtheorem{mylem}{Lemma}[section]

\newtheorem{mycor}{Corollary}[section]

\newtheorem*{mylem3.1}{Lemma 3.1*}

\usepackage{fancyhdr}

\usepackage{color}

\theoremstyle{remark}

\numberwithin{equation}{section}

 \allowdisplaybreaks[4]
\begin{document}

\title{The Restricted Isometry Property of Block Diagonal Matrices Generated by $\varphi$-Sub-Gaussian Variables}

\pagestyle{fancy}

\fancyhf {} 


\fancyhead[CO]{\footnotesize  SUPREMA OF CHAOSES AND THE R.I.P.}

\fancyhead [CE]{\footnotesize Y. CHEN, G. DAI, K.DING}

\fancyhead [LE]{\thepage}
\fancyhead [RO]{\thepage}
\renewcommand{\headrulewidth}{0mm}

\author{Yiming Chen}
\address{Shandong University,  Jinan,  250100, China.}
\email{chenyiming960212@mail.sdu.edu.cn}

\author{Guozheng Dai$^*$}
\address{Zhejiang University, Hangzhou, 310058,  China.}
\email{11935022@zju.edu.cn}

\author{Kaiti Ding}
\address{University of California, Santa Barbara, America.}
\email{kaitiding@gmail.com}




\date{}


\begin{abstract}
In this paper, we prove the restricted isometry property of block diagonal random matrices with elements from $\varphi$-sub-Gaussian variables, which extends the previously known results for the sub-Gaussian case.
A crucial ingredient of our proof is an improved uniform Hanson-Wright deviation inequality, which should be of independent interest.
\end{abstract}

\maketitle

\section{Introduction }
Compressed sensing is a method for reconstructing sparse vectors from incomplete data, as described in \cite{Tao_recovery} and \cite{CJTao_recovery}. Specifically, any \(d\)-dimensional vector \(\mathbf{x}\) with at most \(s\) non-zero elements can be accurately recovered from \(O(s \log (d / s))\) non-adaptive measurements, which satisfy \(\mathbf{y} = \mathbf{B x}\), where \(\mathbf{B} \in \mathbb{C}^{m \times d}\) and \(m \ll d\). This is under the assumption that the measurement matrix \(\mathbf{B}\) satisfies certain structural conditions, and efficient algorithms are used for the reconstruction. This method has been widely applied in the fields of signal and image processing, where the commonly used measurement matrices are random matrices, simulating the process of acquiring linear data.


The Restricted Isometry Property (RIP) \cite{Tao_Rrecovery,Tao_Crecovery}  emerges as a standard analytical instrument to examine the efficacy of a measurement matrix in capturing information pertaining to sparse signals.  It is the feasibility support for the analysis of various signal reconstruction algorithms including $L_1$ minimization, greedy pursuit, and other types of iterative algorithms. Many types of random matrices, including Gaussian and Rademacher matrices, obey the RIP with optimal scaling behavior. These random matrices are referred as unstructured random matrices. Meanwhile, considering that in most engineering applications, most of the structural characteristics of the measurement system are predetermined by the current application, there are also many works that consider the RIP of structured random matrices, such as Partial Random Circulant Matrices and Time-Frequency Structured Random Matrices. 

In this paper, we conduct an exploration of a special class of structured random matrices, which lie at the intersection between random and highly organized measurement sets. Specifically, we focus on block-diagonal measurement matrices, where the blocks are independently generated dense generalized sub-Gaussian random matrices. These measurement models have wide applications in various fields, such as Distributed Compressed Sensing (DCS) and the Multiple Measurement Vector (MMV) framework. In the MMV framework, it is assumed that the low-complexity structure of the signal remains consistent over time, thereby enabling the capture of multiple independent observations of the signal.

The core technology to address this issue lies in representing the RIP of the random matrix as an upper bound of a chaos process (uniform Hanson-Wright type inequality), which is a classical approach for showing the RIP of structured random matrices.  Based on the Talagrand's generic chaining, Tropp et al. derived the RIP for Time-Frequency Structured Random Matrices and Partial Random Circulant Matrices. Krahmer et al. \cite{Rauhut_CPAM} show an improved uniform Hanson-Wright type deviation inequality  through a more refined chaining method, leading to better RIP constants for Time-Frequency Structured Random Matrices and Partial Random Circulant Matrices. Similarly, the properties of the RIP for block diagonal random matrices have developed rapidly. Eftekhari et al. \cite{Eftekhari_yap} obtained the RIP properties of block random matrices where each block on the main diagonal is itself a sub-Gaussian random matrix, and they have made in-depth extensions to certain sparse matrices and random convolution systems. Koep et al. \cite{KOEP2022333} established a group-sparse variant of the classical RIP for block diagonal sensing matrices that act on group-sparse vectors, and provided conditions under which sub-Gaussian block diagonal random matrices satisfy this group-RIP with high probability.

It can be noted that the result of the RIP for block-diagonal measurement matrices mainly focus on the case where the elements of the matrix are sub-Gaussian random variables, while there is relatively less work on random matrices with non-Gaussian elements. In this work,  we extend this direction by considering the elements of block random matrices to be a class of  $\varphi$-sub-Gaussian random variables, which will introduced in the following chapters. The key results stem from our derivation of a novel uniform Hanson-Wright type deviation inequality.

 \subsection{Organization of the paper}
  Section 2 will introduce some notations and auxiliary lemmas, mainly encompassing the moment bounds for decoupled chaos variables and the generic chaining. In Section 3, we will prove our main results. Firstly, we will present an improved uniform Hanson-Wright deviation inequality. Then, we will provide a proof for the RIP of block diagonal matrices. 


\section{Preliminaries}

\subsection{Notations}
We first introduce some notations. Denote $\left(\mathrm{E}\left|\xi_1\right|^p\right)^{1 / p}$ by $\left\|\xi_1\right\|_{L_p}$ for a random variable $\xi_1$. Denote by $\Vert \cdot\Vert_{F}$ the Frobenius matrix norm and   $\|\cdot\|_{l_\alpha \rightarrow l_\beta}$ the operator matrix norm from $l_{\alpha}$ space to $l_{\beta}$ space ($\alpha, \beta\ge 1$). Let $f=(f(x), x \in \mathbb{R})$ be a real-valued function. The function $f^*=\left(f^*(x), x \in \mathbb{R}\right)$ defined by the formula
		$
		f^*(x)=\sup_{y \in \mathbb{R}}(x y-f(y)),
		$
		is called the convex conjugate of $f$, also known as the  Young-Fenchel transform of the function $f$. Given a positive number $a$, we let the symbol $\lfloor a\rfloor$ represent the floor function of $a$.
		
		A random variable $\xi$ is $\alpha$-subexponential if its tail probability satisfies for $t\ge 0$,
	 \begin{align}
		\textsf{P}\{\vert \eta_{1}\vert\ge Kt  \}\le ce^{-ct^{\alpha}},\nonumber
	\end{align}
	where $K$ is a parameter and $c$ is a universal constant. We often refer to $\eta_{1}$ as a sub-Gaussian variable when $\alpha =2$. The $\alpha$-subexponential norm of $\eta_{1}$ is defined as follows:
	\begin{align}
		\Vert \eta_{1}\Vert_{\Psi_{\alpha}}:=\inf\{ t>0:\textsf{E}\exp(\frac{\vert \eta_{1}\vert^{\alpha}}{t^{\alpha}})\le 2\}.\nonumber
	\end{align}
We denote $\xi\sim \mathcal{W}_{s}(\alpha)$ if $\xi$ is  a symmetric Weibull variable with the scale parameter $1$ and the shape parameter $\alpha$. In particular, $-\log \textsf{P}\{\vert\xi\vert>x  \}=x^{\alpha}, x\ge 0$.
\subsection{$\varphi$-sub-Gaussian variables}
In this section, we introduce the concept of the generalized sub-Gaussian random variables whcih is introduced by Buldygin and Kozachenko \cite{Buldyginbook} called $\varphi$-sub-Gaussian distributions. As a generalization of sub-Gaussian distribution, $\varphi$-sub-Gaussian distribution is widely used in many application scenarios. Such as random Fourier series \cite{Antoninijmaa}, sampling theorem \cite{Kozachenkoieee} and Whittaker-Kotelnikov-Shannon (WKS) approximation \cite{Kozachenkojmaa}. Now, we present its relevant definitions and fundamental properties.
\begin{mydef} A function $\varphi(x), x \in \mathbb{R}$, is called an Orlicz $N$-function if it is continuous, even, convex with $\varphi(0)=0$ and monotonically increasing in the set $x>0$, satisfying\\
(1) $\varphi(x)/x\rightarrow 0$, when $x\rightarrow0$; \\
(2) $\varphi(x) / x \rightarrow \infty$, when $x \rightarrow \infty$.

	\end{mydef}
	\begin{mydef} 
		For an Orlicz $N$-function $\varphi$ satisfying $Q$-condition, a zero-mean random variable $\xi$ obeys $\varphi$-sub-Gaussian distribution, if there exists a constant $a\geq 0$ such that the inequality 
		\begin{equation}\label{defin}
			\textsf{E} \exp (\lambda \xi) \leq \exp \left(\varphi\left(\lambda a \right)\right),
		\end{equation}
		holds for all $\lambda \in \mathbb{R}$.
	\end{mydef}
A random variable that obeys $\varphi$-sub-Gaussian distribution is also called a $\varphi$-sub-Gaussian random variable. Let $Sub_{\varphi}(\Omega)$ denotes the space of $\varphi$-sub-Gaussian random variables. For a metric space $T$, a stochastic  process, $(X_{t})_{t\in T}$, is called $\varphi$-sub-Gaussian process if $X_t \in Sub_{\varphi}(\Omega)$ for all $t \in T$. 
	
	\begin{mylem}\label{l2}
		For any random variable $\xi \in \operatorname{Sub}_{\varphi}(\Omega)$, 
		\begin{equation}\label{2p}
			\textsf{E} \exp \{\lambda \xi\} \leq \exp \left\{\varphi\left(\lambda \tau_{\varphi}(\xi)\right)\right\}, \quad \lambda \in \mathbb{R}
		\end{equation}
		holds, where $$
		\tau_{\varphi}(\xi):=\sup _{\lambda \neq 0} \frac{\varphi^{(-1)}(\log \textsf{E} \exp \{\lambda \xi\})}{|\lambda|},
		$$
		and  $\varphi^{(-1)}(\cdot)$ denotes the inverse function of $\varphi(\cdot)$.

	\end{mylem}
	
	\begin{mylem}
		$Sub_{\varphi}(\Omega)$  is a Banach space with respect to the norm 
		\begin{equation}
			\tau_{\varphi}(X)=\inf \{a \geq 0: \textsf{E} \exp \{\lambda X\} \leq \exp \{\varphi(a \lambda)\}, \,\,\text{for all}\,\,\lambda \in \mathbb{R}\}.
		\end{equation}
		The norm above is equivalent to the one in Lemma \ref{l2}, we recommend referring to \cite{Buldyginbook} for proof, with the norm above, for any $u>0$, the  $\varphi$-sub-Gaussian process $(X_{t})_{t\in T}$ satisfies the following incremental inequality:
		\begin{equation}\label{2.1}
			\textsf{P}\left(|X_t-X_s| \geqslant u \tau_\varphi(X_t- X_s)\right) \leqslant 2 \exp \left\{-\varphi^*(u)\right\},
		\end{equation}
		where $t,s\in T$.
		
	\end{mylem}
	\begin{myrem}
		Centered Gaussian  process  $(X_{t})_{t\in T}$ is an important example of  $\varphi$-sub-Gaussian process,  where $\varphi(x)=x^2 / 2$ and $\tau_{\varphi}(X_t)=\left(\textsf{E}|X_t|^2\right)^{1 / 2}$.
	\end{myrem}

\begin{mydef}
We say that a function $\varphi(\cdot)$ satisfies the $\Delta_2$-condition if it meets that for $x> 1$, 
$$
\varphi_i(2 x) \leq 2 \varphi_i(x).
$$
\end{mydef}

\begin{mydef} A random vector $X$ is called a $\varphi$-sub-Gaussian random vector if all one-dimensional marginals of $X$, i.e., the random variables $\langle X, x\rangle$ for any $x \in \mathbb{R}^n$, are $\varphi$-sub-Gaussian. The $\varphi$-sub-Gaussian norm of $X$ is defined as
$$
\tau_{\varphi}(X):=\sup _{x \in \mathrm{S}^{n-1}} \tau_{\varphi}(\langle X, x\rangle),
$$
where $S^{n-1}$ denotes the Euclidean unit sphere in $\mathbb{R}^n$.
\end{mydef}

\subsection{The Generic Chaining}
In this section, we will introduce the basic definitions and properties of the classical generic chaining method, pioneered and refined by Talagrand in his work \cite{Talagrand_chaining_book}.

For a metric space $(T, d)$, we call a sequence of subsets $\left\{T_n: n \geq 0\right\}$ of $T$ an admissible sequence if for every $n \geq 1,\left|T_n\right| \leq 2^{2^n}$ and $\left|T_0\right|=1$. For any $0<\alpha<\infty$, the Talagrand's $\gamma_\alpha$-functional of $(T, d)$ is defined as follows
$$
\gamma_\alpha(T, d)=\inf \sup _{t \in T} \sum_{r \geq 0} 2^{n / \alpha} d\left(t, T_n\right),
$$
where the infimum is taken concerning all admissible sequences of $T$. 

Given a set $T$ and an admissible sequence $(\mathcal{T}_{n})_{n\ge 0}$, denote  the unique element of $\mathcal{T}_{n}$ which contains $t$ by $T_{n}(t)$ and  the diameter of the set $T$ by $\Delta_{d}(T)$. Let
\begin{align}
	\gamma_{\alpha}^{\prime}(T, d)=\inf_{\mathcal{T}}\sup_{t\in T}\sum_{n\ge 0}2^{n/\alpha}\Delta_{d}(T_{n}(t)),\nonumber
\end{align}
where the infimum is taken over all admissible partitions $\mathcal{T}$ of $T$. Talagrand \cite{Talagrand_annals_prob} showed $\gamma_{\alpha}(T, d)\le \gamma^{\prime}_{\alpha}(T, d)\le C(\alpha)\gamma_{\alpha}(T, d)$.

For a fixed radius $u>0$, we denote the covering number of $T$ by $N(T, d, u)$, that is, the minimum number of balls with radius $u$ required to cover the set $T$. One can bound the $\gamma_\alpha$-functionals with such covering numbers. In particular, we have
$$
\gamma_\alpha(T, d) \lesssim \alpha \int_0^{\infty}(\log N(T, d, u))^{1 / \alpha} d u.
$$

The bound of the above inequality is the well-known Dudley integral. The details of the inequality above can be be traced back to in \cite{Talagrand_chaining_book} for the case of $\alpha=2$, one can also find in \cite{Dirksen_EJP} for the case of $\alpha >0$. In this work, we focus on the $\varphi$-sub-Gaussian process, whose Talagrand-type $\gamma$ functional have been introduced and well studied by Chen et al.\cite{chen_li_liu_wang}. 

\begin{mydef}
	For an Orlicz $N$-function $\varphi$ satisfying $Q-$condition, then for $1 \leq p<\infty$ distribution-dependent Talagrand-type $\gamma$-functional 
	is defined by 
	\begin{eqnarray}\label{gg1}
		\gamma_{\varphi,p}(T, d)=\inf\sup _{t \in T} \sum_{n\geq k}^{\infty} \varphi^{*(-1)}(2^{n}) \Delta \left(A_n(t)\right),
	\end{eqnarray}	
	where the infimum is taken over all admissible sequences and $k:=\left\lfloor\frac{\log (p)}{\log (2)} \right\rfloor$, $A_n(t)$ is the unique element of $\mathcal{T}_n$ which contains $t$, $\Delta$ is the diameter of the $A_n(t)$ with respect to the norm $\|\cdot\|_{\tau_{\varphi}}$.
\end{mydef}

The above functional is the result for the p-th moment, and when $\varphi(x)=\frac{x^2}{2}$, we denote $\gamma_{\varphi,p}(T, d)$ by $\gamma_{2,p}(T, d)$. 

\subsection{Order 2 chaos variables}

In this subsection, we introduce some properties of the order 2 chaos variable $\sum_{i, j}a_{ij}\xi_{i}\xi_{j}$. Here $(a_{ij})$ is a fixed matrix and $\{\xi_{i}\}$ is a squence of independent random variables. The first is the following decoupling inequality.

\begin{mylem}[Proposition 3.1 in \cite{Dai_Su_Wang_U}]\label{Lem_decoupling}
	Let $F$ be a convex function satisfying $F(x)=F(-x), \forall x\in \mathbb{R}$.  Assume $\xi=(\xi_{1}, \cdots, \xi_{n})^\top$ and $ \eta=(\eta_{1}, \cdots, \eta_{n})^\top$ are random vectors with independent centered entries. Let for any $t>0$ and $i\ge 1$, the independent random variables $\xi_{i}, \eta_{i}$ satisfy for  some $c\ge 1$
	\begin{align}
		\textsf{P}\{ \xi_{i}^{2}\ge t \}\le c\textsf{P}\{ c\vert \eta_{i}\tilde{\eta}_{i}\vert\ge t \},\nonumber
	\end{align}
	where $\tilde{\eta}_{i}$ is an independent copy of $\eta_{i}$. 
	Then, there exists a constant $C$ depending only on $c$ such that
	\begin{align}
		\textsf{E}\sup_{A\in \mathcal{A}}F\big(\xi^\top A\xi-\textsf{E}\xi^\top A\xi\big)\le \textsf{E}\sup_{A\in \mathcal{A}}F\big( C\eta^\top A\tilde{\eta}\big)\nonumber
	\end{align}
	where  $\mathcal{A}$ is a set of $n\times n$ fixed matrices.
\end{mylem}

Then, we present an upper bound for the $p$-th moment of some bilinear random variables. 

\begin{mylem}[Remark 2.1 in \cite{Dai_Su_Wang_U}]\label{Lem_bound_pth_decoupled}
	Let $\{\xi_{i}, i\le n \}$ be a sequence of i.i.d. variables with distribution $\mathcal{W}_{s}(\alpha)$, $0<\alpha\le 2$. Let $A=(a_{ij})_{n\times n}$ be a fixed symmetric matrix and $\{\tilde{\xi}_{i}, i\le n \}$ be independent copies of $\{\xi_{i}, i\le n\}$. Then, we have for $p\ge 2$
	\begin{align}
		\big\Vert\sum_{i, j}a_{ij}\xi_{i}\tilde{\xi}_{j} \big\Vert_{L_{p}}\lesssim_{\alpha} p^{1/2}\Vert A\Vert_{F}+p^{2/\alpha}\Vert A\Vert_{l_{2}\to l_{2}}.\nonumber
	\end{align}
\end{mylem}

\subsection{Tails and Moments}
In this subsection, we shall present the following property about the relationship between the tail and the moment of random variables.

\begin{mylem}[Lemma 2.1 in \cite{Dai_Su_Wang_U}]\label{Lem_Moments_2}
	Assume that a random variable $\xi$ satisfies for $p\ge p_{0}$
	\begin{align}
		\Vert \xi\Vert_{L_p}\le \sum_{k=1}^{m}C_{k}p^{\beta_{k}}+C_{m+1},\nonumber
	\end{align}
	where $C_{1},\cdots, C_{m+1}> 0$ and $\beta_{1},\cdots, \beta_{m}>0$. Then we have for any $t>0$,
	\begin{align}
		\textsf{P}\big\{ \vert \xi\vert>e(mt+C_{m+1}) \big\}\le e^{p_{0}}\exp\Big(-\min\big\{\big(\frac{t}{C_{1}}\big)^{1/\beta_{1}},\cdots, \big(\frac{t}{C_{m}}\big)^{1/\beta_{m}}\big\}\Big)\nonumber
	\end{align}
	and
	\begin{align}
		\textsf{P}\Big\{\vert \xi\vert>e \big(\sum_{k=1}^{m}C_{k}t^{\beta_{k}}+C_{m+1}\big)  \Big\}\le e^{p_{0}}e^{-t}.\nonumber
	\end{align}
\end{mylem}

\section{Main Results and Proofs}

\subsection{Uniform Hanson-Wright type deviation inequalities}

Let $\xi_{1}, \cdots,\xi_{n}$ be a sequence of independent random variables with mean 0. Let $A=(a_{ij})_{n\times n}$ be a fixed matrix. The concentration properties of $\sum a_{ij}\xi_{i}\xi_{j}$ is a classic topic in probability. A well-known result is due to Hanson and Wright, claiming that if $\xi_{i}$ are sub-Gaussian (2-subexponential) variables satisfying $\max_{i} \Vert \xi_{i}\Vert_{\Psi_{2}}\le L$ and $A$ is symmetric, then for all $t\ge 0$ (the following modern version was in \cite{Rudelson_ecp})
\begin{align}
		\textsf{P}\big\{\vert \sum a_{ij}\xi_{i}\xi_{j}-\textsf{E}\sum a_{ij}\xi_{i}\xi_{j}\vert \ge t  \big\}\le 2\exp\Big(-c\min\Big\{\frac{t^{2}}{L^{4}\Vert A\Vert_{F}^{2}}, \frac{t}{L^{2}\Vert A\Vert_{l_{2}\to l_{2}}}  \Big\}   \Big).\nonumber
\end{align}

An interesting extension is to consider the $\alpha$-subexponential case. In particular, Sambale \cite{Sambale} proved when $\xi_{i}$ are $\alpha$-subexponential variables ($0<\alpha\le 2$) and $A$ is symmetric
\begin{align}\label{Eq_Hanson_Wright_alpha_subexponential}
	&\textsf{P}\big\{\vert \sum a_{ij}\xi_{i}\xi_{j}-\textsf{E}\sum a_{ij}\xi_{i}\xi_{j}\vert \ge t  \big\}\nonumber\\
	\le& 2\exp\Big(-c\min\Big\{\frac{t^{2}}{L^{4}\Vert A\Vert_{F}^{2}}, \Big(\frac{t}{L^{2}\Vert A\Vert_{l_{2}\to l_{2}}}\Big)^{\alpha/2}  \Big\}   \Big).
\end{align} 

Exploring the properties of $\sup_{A\in\mathcal{A}}\sum a_{ij}\xi_{i}\xi_{j}$ is also an interesting extension, where $\mathcal{A}$ is a family of fixed matrices. We refer  the interested readers to \cite{Klochkov_EJP,Adamczak_AIHP_2,Talagrand_invention} for some celebrated results in this direction. 

Recently, Dai et al. \cite{Dai_Su_Wang_U, Dai_Su} showed the following uniform Hanson-Wright type deviation inequalities, which extend a well-known result \cite{Rauhut_CPAM} from the sub-Gaussian case to the $\alpha$-subexponential case.  Before introducing their result, we first give some notations. Let $\mathcal{A} $ be a matrix set. Then the radii of $\mathcal{A}$ w.r.t. the Frobenius and the operator norm are defined as
$$
M_F(\mathcal{A})=\sup _{\boldsymbol{A} \in \mathcal{A}}\|\boldsymbol{A}\|_{\mathrm{F}} \quad \text { and } \quad M_{l_\alpha \rightarrow l_{\beta}}(\mathcal{A})=\sup _{\boldsymbol{A} \in \mathcal{A}}\|\boldsymbol{A}\|_{l_\alpha \rightarrow l_\beta},
$$
Let
\begin{align}
	\Gamma(\alpha, \beta, \mathcal{A})=\gamma_{2}(\mathcal{A}, \Vert\cdot\Vert_{l_{2}\to l_{2}})+\gamma_{\alpha}(\mathcal{A}, \Vert\cdot\Vert_{l_{2}\to l_{\beta}})\nonumber
\end{align}
and 
\begin{align}
	U_{1}(\alpha, \beta)=& \Gamma(\alpha, \beta,  \mathcal{A}) \big(\Gamma(\alpha, \beta,  \mathcal{A})+M_{F}(\mathcal{A}) \big),\nonumber\\
	U_{2}(\alpha, \beta)=&M_{l_{2}\to l_{2}}(\mathcal{A})\Gamma(\alpha, \beta, \mathcal{A})+\sup_{A\in\mathcal{A}}\Vert A^\top A\Vert_{F},\nonumber\\
	U_{3}(\alpha, \beta)=&M_{l_{2}\to l_{\beta}}(\mathcal{A})\Gamma(\alpha, \beta, \mathcal{A}) .\nonumber
\end{align}

Let $\xi_{i}$ be independent centered $\alpha$-subexponential variables ($0<\alpha\le 2$) and $L=L(\alpha)=\max_{i}\Vert\xi_{i}\Vert_{\Psi_{\alpha}}$. Dai et al. (see Corollary $1.2^{*}$ in \cite{Dai_Su_Wang_U} and Theorem 1.1 in \cite{Dai_Su}) proved for $t\ge 0$
\begin{align}
	&\textsf{P}\Big\{ \sup_{A\in \mathcal{A}}\big\vert \Vert A\xi\Vert_{2}^{2}-\textsf{E}\Vert A\xi\Vert_{2}^{2}  \big\vert>C(\alpha)L^{2}(U_{1}(\alpha, \alpha^{*})+t)\Big\}\nonumber\\
	\le& C_{1}(\alpha)\exp\Big(-\min\Big\{(\frac{t}{U_{2}(\alpha, \alpha^{*})})^{2},( \frac{t}{U_{3}(\alpha, \alpha^{*})})^{\alpha}, (\frac{t}{M_{l_{2}\to l_{2}}^{2}(\mathcal{A})})^{\alpha/2}\Big\}\Big).\nonumber
\end{align}
Here, $\mathcal{A}$ is a family of $m\times n$ fixed matrices, $\xi=(\xi_{1},\cdots, \xi_{n})^{\top}$, and $\alpha^{*}$ is defined as follows:
\[\alpha^{*}=\left\{\begin{array}{ll}
	\frac{\alpha}{\alpha-1}&1<\alpha\le 2,\\
	\infty&
	0<\alpha\le 1.
\end{array}\right.\]
In fact, one would expect from \eqref{Eq_Hanson_Wright_alpha_subexponential} that $U_{2}(\alpha, \alpha^{*})$ in the above deviation inequality can be replaced by the smaller factor $\sup_{A\in\mathcal{A}}\Vert A^\top A\Vert_{F}$. Our first main result shows that this is indeed possible.
\begin{mytheo}\label{Theo_uniform_hansonwright}
	Let $\xi=(\xi_{1},\cdots, \xi_{n})$ be a random vector with independent centered $\alpha$-subexponential variables ($0<\alpha\le 2$) and $\mathcal{A}$ is a family of $m\times n$ fixed matrices. We have for $t\ge 0$
	\begin{align}
		&\textsf{P}\Big\{ \sup_{A\in \mathcal{A}}\big\vert \Vert A\xi\Vert_{2}^{2}-\textsf{E}\Vert A\xi\Vert_{2}^{2}  \big\vert>C(\alpha)L^{2}(U_{1}(\alpha, \alpha^{*})+t)\Big\}\nonumber\\
		\le& C_{1}(\alpha)\exp\Big(-\min\Big\{(\frac{t}{\sup_{A\in \mathcal{A}}\Vert A^{\top}A\Vert_{F}})^{2}, (\frac{t}{M_{l_{2}\to l_{2}}^{2}(\mathcal{A})})^{\alpha/2}\Big\}\Big),\nonumber
	\end{align}
where $U_{1}(\alpha, \alpha^{*})$ and $L$ are defined as above.
\end{mytheo}

\begin{proof}
Without loss of generality, we assume $L=1$.	Let $\eta_{1}, \cdots, \eta_{n}\stackrel{i.i.d.}{\sim}\mathcal{W}_{s}(\alpha)$. Then, Lemma \ref{Lem_decoupling} yields for $p\ge 1$
	\begin{align}\label{Eq_main_decoupling}
		\Big\Vert \sup_{A\in \mathcal{A}}\big\vert \Vert A\xi\Vert_{2}^{2}-\textsf{E}\Vert A\xi\Vert_{2}^{2}  \big\vert \Big\Vert_{L_{p}}\lesssim_{\alpha} \Big\Vert \sup_{A\in \mathcal{A}}\big\vert\eta^\top A^{\top} A\tilde{\eta}  \big\vert \Big\Vert_{L_{p}},
	\end{align}
where $\eta=(\eta_{1}, \cdots,\eta_{n})^{\top}$ and $\tilde{\eta}$ is an independent copy of $\eta$.

 For convienence, denote by $d_{2}$ and $d_{\alpha^{*}}$ the distance on $\mathcal{A}$ induced by $\Vert\cdot\Vert_{l_{2}\to l_{2}}$ and $\Vert\cdot\Vert_{l_{2}\to l_{\alpha^{*}}}$. Recalling the definition of $\gamma^{\prime}_{\alpha}$ functionals, we can select two admissible sequences of partitions $\mathcal{T}^{(1)}=(\mathcal{T}^{(1)}_{n})_{n\ge 0}$ and $\mathcal{T}^{(2)}=(\mathcal{T}^{(2)}_{n})_{n\ge 0}$ of $\mathcal{A}$ such that
\begin{align}
	\sup_{A\in \mathcal{A}}\sum_{n\ge 0}2^{n/2}\Delta_{d_{2}}(T_{n}^{(1)}(A))\le 2\gamma_{2}^{\prime}(\mathcal{A}, d_{2}),\quad \sup_{A\in \mathcal{A}}\sum_{n\ge 0}2^{n/\alpha}\Delta_{d_{\alpha^{*}}}(T_{n}^{(2)}(A))\le 2\gamma_{\alpha}^{\prime}(\mathcal{A}, d_{\alpha^{*}}).\nonumber
\end{align}
Let $\mathcal{T}_{0}=\{\mathcal{A}\}$ and 
\begin{align}
	\mathcal{T}_{n}=\{T^{(1)}\cap T^{(2)}:  T^{(1)}\in \mathcal{T}^{(1)}_{n-1}, T^{(2)}\in \mathcal{T}^{(2)}_{n-1}   \},\quad n\ge 1\nonumber.
\end{align}
Then, $\mathcal{T}=(\mathcal{T}_{n})_{n\ge 0}$ is a sequence of  increasing partitions  and 
\begin{align}
	\vert \mathcal{T}_{n}\vert\le \vert \mathcal{T}^{(1)}_{n-1}\vert\vert \mathcal{T}^{(2)}_{n-1}\vert\le 2^{2^{n-1}}2^{2^{n-1}}=2^{2^{n}}.\nonumber
\end{align}
We next define a subset $\mathcal{A}_{n}$ of $\mathcal{A}$ by selecting exactly one point from each $T\in \mathcal{T}_{n}$. By this means, we build an admissible sequence $(\mathcal{A}_{n})_{n\ge 0}$ of subsets of $\mathcal{A}$. Let $\pi=\{ \pi_{r}, r\ge 0\}$ be a sequence of functions $\pi_{n}: \mathcal{A}\to \mathcal{A}_{n}$ such that $\pi_{n}(A)=\mathcal{A}_{n}\cap T_{n}(A)$, where $T_{n}(A)$ is the element of $\mathcal{T}_{n}$ containing $A$. Let $l$ be the largest integer satisfying $2^{l}\le p$.
Then, we have for $p\ge 1$ (see Proposition 3.1 in \cite{Dai_Su} and Lemma 3.1 in \cite{Dai_Su_Wang_U})
\begin{align}\label{Eq_decomp}
	\big\Vert \sup_{A\in \mathcal{A}}\vert \eta^\top A^\top A\tilde{\eta}\vert\big\Vert_{L_{p}}\lesssim_{\alpha} 	\big\Vert \sup_{A\in\mathcal{A}}\Vert A\eta\Vert_{2} \big\Vert_{L_{p}}\cdot\Gamma(\alpha, \alpha^{*}, \mathcal{A})+\sup_{A\in \mathcal{A}}\big\Vert  \eta^\top A^\top A\tilde{\eta}\big\Vert_{L_{p}}.\nonumber
\end{align}

Note that
\begin{align}
	\big\Vert \sup_{A\in\mathcal{A}}\Vert A\eta\Vert_{2} \big\Vert_{L_{p}}&=\Big\Vert \sup_{A\in\mathcal{A}}\big(\Vert A\eta\Vert^{2}_{2}-\textsf{E}\Vert A\eta\Vert^{2}_{2}+\textsf{E}\Vert A\eta\Vert^{2}_{2}\big)^{1/2} \Big\Vert_{L_{p}}\nonumber\\
	&\le\Big\Vert \sup_{A\in\mathcal{A}}\big\vert\Vert A\eta\Vert^{2}_{2}-\textsf{E}\Vert A\eta\Vert^{2}_{2}\big\vert\Big\Vert^{1/2}_{L_{p}}+M_{F}(\mathcal{A}).\nonumber
\end{align}
Meanwhile, Lemma \ref{Lem_bound_pth_decoupled} yields for $p\ge 1$
\begin{align}
	\sup_{A\in \mathcal{A}}\big\Vert  \eta^\top A^\top A\tilde{\eta}\big\Vert_{L_{p}}\lesssim p^{1/2}\sup_{A\in\mathcal{A}}\Vert A^{\top}A\Vert_{F}+p^{2/\alpha}\sup_{A\in\mathcal{A}}\Vert A^{\top}A\Vert_{l_{2}\to l_{2}}.\nonumber
\end{align}
 Hence, we have by Lemma \ref{Lem_bound_pth_decoupled}
 \begin{align}
 		\big\Vert \sup_{A\in \mathcal{A}}\vert \eta^\top A^\top A\tilde{\eta}\vert\big\Vert_{L_{p}}\lesssim_{\alpha}& 	\Big(\big\Vert \sup_{A\in \mathcal{A}}\vert \eta^\top A^\top A\tilde{\eta}\vert\big\Vert_{L_{p}}^{1/2}+M_{F}(\mathcal{A}) \Big)\cdot\Gamma(\alpha, \alpha^{*}, \mathcal{A})\nonumber\\
 		&+p^{1/2}\sup_{A\in\mathcal{A}}\Vert A^{\top}A\Vert_{F}+p^{2/\alpha}\sup_{A\in\mathcal{A}}\Vert A^{\top}A\Vert_{l_{2}\to l_{2}}.\nonumber
 \end{align}
Then, \eqref{Eq_main_decoupling} yields that
\begin{align}
	\Big\Vert \sup_{A\in \mathcal{A}}\big\vert \Vert A\xi\Vert_{2}^{2}-\textsf{E}\Vert A\xi\Vert_{2}^{2}  \big\vert \Big\Vert_{L_{p}}\lesssim_{\alpha}& U_{1}(\alpha, \alpha^{*})+p^{1/2}\sup_{A\in\mathcal{A}}\Vert A^\top A\Vert_{F}\nonumber\\
	&+p^{2/\alpha}\sup_{A\in\mathcal{A}}\Vert A^\top A\Vert_{l_{2}\to l_{2}}.\nonumber
\end{align}
 The desired result follows from Lemma \ref{Lem_Moments_2}.

\end{proof}

\begin{mycor}\label{t4.1}
	Let $\xi=(\xi_{1},\cdots, \xi_{n})$ be a random vector with independent centered $\varphi$-sub-Gaussian variables and $\mathcal{A}$ is a family of $m\times n$ fixed matrices. We have for $t\ge 0$
	\begin{align}
		&\textsf{P}\Big\{ \sup_{A\in \mathcal{A}}\big\vert \Vert A\xi\Vert_{2}^{2}-\textsf{E}\Vert A\xi\Vert_{2}^{2}  \big\vert>CL_{1}^{2}(U_{1}(1, \infty)+t)\Big\}\nonumber\\
		\le& C_{1}\exp\Big(-\min\Big\{(\frac{t}{\sup_{A\in \mathcal{A}}\Vert A^{\top}A\Vert_{F}})^{2}, (\frac{t}{M_{l_{2}\to l_{2}}^{2}(\mathcal{A})})^{1/2}\Big\}\Big),\nonumber
	\end{align}
	where $U_{1}(1, \infty)$ is defined as above and $L_{1}=\max_{i}\tau_{\varphi}(\xi_{i})$.
\end{mycor}
\begin{proof}
	Let $\eta_{1},\cdots, \eta_{n}\stackrel{i.i.d.}{\sim}\mathcal{W}_{s}(1)$. Let $\tilde{\eta}_{i}$ be an independent copy of $\eta_{i}$ for $i\le n$. Then, by the property of N-function, there exist a constant $t_0$ such that $\varphi^*(t)>t$ for $t>t_0$, then we obtain that for $t>t_0$,we have 
	\begin{align}
		\textsf{P}\{\xi^{2}_{i}\ge 4t  \}&\le \exp(-\varphi^{*}(\frac{2\sqrt{t}}{K}))\le \exp(-\frac{2\sqrt{t}}{K})\nonumber\\
		&=\textsf{P}\big\{ \vert\eta_{i}\vert\ge \frac{\sqrt{t}}{K}   \big\}^{2}\le\textsf{P}\{c\vert\eta_{i}\tilde{\eta}_{i}\vert\ge t  \}.\nonumber
	\end{align}
As for $t\le t_{0}$, a trival bound is as follows:
\begin{align}
     \textsf{P}\{\xi^{2}_{i}\ge 4t  \}\le 1=\frac{\textsf{P}\{\vert\eta_{i}\tilde{\eta}_{i}\vert\ge t_{0}  \}}{\textsf{P}\{\vert\eta_{i}\tilde{\eta}_{i}\vert\ge t_{0}  \}}\le \frac{\textsf{P}\{\vert\eta_{i}\tilde{\eta}_{i}\vert\ge t  \}}{\textsf{P}\{\vert\eta_{i}\tilde{\eta}_{i}\vert\ge t_{0}  \}}.\nonumber
\end{align}
Let $c_{1}=\max\{4c, 1/\textsf{P}\big\{\vert\eta_{i}\tilde{\eta}_{i}\vert\ge t_{0}  \}  \big\}$. Then, we have for $t>0$
\begin{align}
	\text{P}\{\xi^{2}_{i}\ge   t\}\le c_{1}\text{P}\{c_{1}\vert\eta_{i}\tilde{\eta}_{i}\vert\ge t  \}.\nonumber
\end{align}
The desired result follows from Lemma \ref{Lem_decoupling} and Theorem \ref{Theo_uniform_hansonwright}.
\end{proof}

\subsection{The group-RIP for block diagonal matrices}
Consider a vector $\mathbf{x} \in \mathbb{C}^D$, which we categorize the observed vectors into $L$ groups based on the dimensions, i.e. $D=\sum_{i=1}^L d_i$. For simplicity, we assume that the dimension $D$ is exactly divisible by the number of groups $L$, i.e. $D$ equals $d$ times $L$ with $d \in \mathbb{N}$ and $d_i=d$ for all $i \in[L]$, and we partition the set of $D$ into $G$ groups as follows, we also refer to \cite{KOEP2022333} for more details.

\begin{mydef}
	A set \(\mathcal{S} = \{\mathcal{S}_1, \ldots, \mathcal{S}_G\}\) of subsets \(\mathcal{S}_i \subseteq [D] := \{1, \ldots, D\}\) is referred to as a group partition of \([D]\) if \(\mathcal{S}_i \cap \mathcal{S}_j = \emptyset\) for all \(i \neq j\), and \(\bigcup_{i=1}^G \mathcal{S}_i = [D]\). Obviously, the elements within $\mathcal{S}_i$ do not need to be consecutive indices, and it is not assumed that the sizes of the individual sets are uniform.

\end{mydef}

Let \( g_i = |\mathcal{S}_i| \) represent the size of each group \( \mathcal{S}_i \), and define \( g = \max_{i \in [G]} g_i \) as the largest group size among all groups. To formalize group sparsity, we introduce the following notation. Consider \(\mathbf{x}_{\mathcal{S}_i} \in \mathbb{C}^D\) as the vector that restricts \(\mathbf{x}\) to the indices in \(\mathcal{S}_i\), where for each \(j \in [D]\), the component is defined by \((\mathbf{x}_{\mathcal{S}_i})_j = x_j \cdot \mathbf{1}_{\{j \in \mathcal{S}_i\}}\). Here, \(\mathbf{1}_{\{j \in \mathcal{S}_i\}}\) is the indicator function.

A signal \(\mathbf{x}\) is termed \(s\)-group-sparse if it is non-zero in at most \(s\) groups. This means there exists a subset \(S \subset [G]\) with \(|S| \leq s\) such that \(\mathbf{x} = \sum_{i \in S} \mathbf{x}_{\mathcal{S}_i}\). Furthermore, we define a set of mixed norms that are well-suited for the applications considered in this paper.

\begin{mydef} Consider $\mathbf{x} \in \mathbb{C}^D$, for any $p \geq 1$, we introduce the group \( \ell_{\mathcal{S}, p} \)-norm on \( \mathbb{C}^D \) defined by 

\[
\|\mathbf{x}\|_{{\mathcal{S}, p}} = \left( \sum_{i=1}^G \left\| \mathbf{x}_{\mathcal{S}_i} \right\|_2^p \right)^{1/p},
\]

\end{mydef}

When considering the case where $p=0$, the $\|\cdot\|_{\mathcal{S}, 0}$ represents the pseudonorm corresponding to the group $\ell_{\mathcal{S}, 0}$ which is essentially a count of the number of groups that contain non-zero components of the vector,
$
\|\mathbf{x}\|_{\mathcal{S}, 0}:=\left|\left\{i \in[G]: \mathbf{x}_{\mathcal{S}_i} \neq \mathbf{0}\right\}\right| .
$

Then $
\Sigma_{\mathcal{S}, s}=\left\{\mathbf{x} \in \mathbb{C}^D:\|\mathbf{x}\|_{\mathcal{S}, 0} \leq s\right\}
$
is defined the set of $s$-group-sparse vectors w.r.t. the group partition $\mathcal{S}$. In particular, 
\[
\sigma_s(\mathbf{x})_{\mathcal{S}, 1} = \inf_{\mathbf{z} \in \Sigma_{\mathcal{S}, s}} \|\mathbf{x} - \mathbf{z}\|_{\mathcal{S}, 1}
\]
represents the optimal $s$-term group approximation error. A key aspect is that this error diminishes quickly as $s$ increases, which is crucial for describing the compressibility characteristics of block diagonal matrices.

As previously mentioned, we observe an \( s \)-group-sparse or compressible signal \( \mathbf{x} \in \Sigma_{\mathcal{I}, s} \) using a block diagonal measurement matrix \( \mathbf{B} \) composed of \( L \) blocks \( \Phi_l \):

\[
\mathbf{B} = \operatorname{diag}(\Phi_1, \Phi_2, \ldots, \Phi_L).
\]

Assuming \( \mathbf{x} \) is known only through its basis expansion \( z = \Psi \mathbf{x} \) in a unitary basis \( \Psi \in U(D) := \{ U \in \mathbb{C}^{D \times D} : U^{*} U = I_{D} \} \), the measurement model simplifies to:

\begin{equation}\label{koe1}
y = \mathbf{B} z = \mathbf{B} \Psi x. 
\end{equation}

The main objective of this paper is to establish sufficient conditions that guarantee the stable and robust reconstruction of group-sparse signals. To achieve this, we focus on formulating an appropriate Restricted Isometry Property (RIP) for block diagonal matrices when they act on group-sparse vectors.

Now, let us introduce the concept of group restricted isometry property (group-RIP), which is a generalization of the renowned RIP, modeled upon the block-sparse RIP initially proposed in \cite{elder_mis}.

\begin{mydef}\label{b2.1}

A matrix \( \mathbf{B} \in \mathbb{C}^{M \times D} \) is said to have the group restricted isometry property (group-RIP) of order \( s \) if there exists a constant \( \delta \in (0, 1) \) such that, for all $\mathbf{x} \in \Sigma_{\mathcal{S}, s}$, the following inequality holds:

	\begin{equation}\label{ko2}
		(1-\delta)\|\mathbf{x}\|_2^2 \leq\|\mathbf{B} \mathbf{x}\|_2^2 \leq(1+\delta)\|\mathbf{x}\|_2^2 \quad \forall \mathbf{x} \in \Sigma_{\mathcal{S}, s}
	\end{equation}

The smallest constant \( \delta_s \leq \delta \) satisfying this inequality is referred to as the group restricted isometry constant (group-RIC) of $\mathbf{B}$.
\end{mydef}

It is noteworthy that the stability and robustness results concerning the group RIP constant have already been established in the seminal work of Eldar and Mishali \cite{elder_mis} as follows, albeit with the necessary condition $\delta_{2 s}<\sqrt{2}-1$ on the group-RIP constant. In the work of Koep et al.\cite{KOEP2022333}, they extended the results of group diagonal random matrices to random matrices where each group consists of sub-Gaussian random variables, and derived their RIP condition.

\begin{mytheo}[\cite{KOEP2022333}] 
Consider a block diagonal random matrix \(\mathbf{B} = \operatorname{diag}\{\boldsymbol{\Phi}_l\} \in \mathbb{R}^{mL \times dL}\), where each block \(\boldsymbol{\Phi}_l\) consists of independent sub-Gaussian random variables with zero mean, unit variance, and sub-Gaussian norm \(\tau\). Let \(\boldsymbol{\Psi} \in \mathrm{U}(dL)\) be a unitary matrix. Then, with probability at least \(1 - \eta\), the scaled matrix \(m^{-1/2} \mathbf{B} \boldsymbol{\Psi}\) satisfies the group restricted isometry property (group-RIP) of order \(s\) with respect to the group partition \(\mathcal{S}\), and the group restricted isometry constant \(\delta_s\) satisfies \(\delta_s \leq \delta\) under certain conditions.

	$$
	m \gtrsim_\tau \delta^{-2}\left[s \mu_{\mathcal{S}}^2 \log (D) \log (s)^2\left(\log (G)+g \log \left(s / \mu_{\mathcal{S}}\right)\right)+\log \left(\eta^{-1}\right)\right],
	$$
	where
	$$
	\mu_{\mathcal{S}}=\mu_{\mathcal{S}}(\Psi):=\min \left\{\sqrt{d} \max _{i \in[D]}\left\|\boldsymbol{\psi}_i\right\|_{\mathcal{S}, \infty}, 1\right\}
	$$
	and $\boldsymbol{\psi}_i \in \mathbb{C}^D$ denotes the $i$-th row of $\boldsymbol{\Psi}$.
\end{mytheo}

We are now in a position to propose our second main result,  the RIP property of block-diagonal random matrices with $\varphi$-sub-Gaussian random variables entries. We will make use of the bound on the suprema of chaos processes first established in this paper to demonstrate that the block diagonal matrix $\mathbf{B} \Psi \in \mathbb{C}^{M \times D}$ satisfies the group restricted isometry property with high probability on the draw of $\mathbf{B}$. The same technique was also employed in \cite{Eftekhari_yap} to prove the canonical restricted isometry property for block diagonal matrices consisting of sub-Gaussian blocks. Koep et al \cite{KOEP2022333} considered an improved version of the bound due to Dirksen \cite{Dirksen_EJP}. We make use of the bound in Corollary \ref{t4.1}.

\begin{mytheo}
Let \(\mathbf{B} = \operatorname{diag}\{\boldsymbol{\Phi}_l\} \in \mathbb{R}^{mL \times dL}\), where each block \(\boldsymbol{\Phi}_l\) consists of independent, zero-mean, unit-variance, \(\varphi\)-sub-Gaussian random variables with sub-Gaussian norm \(\|\cdot\|_{\varphi}\). Let \(\Psi \in \mathrm{U}(dL)\) be a unitary matrix. Then, with probability at least \(1 - \eta\), the scaled matrix \(m^{-1/2} \mathbf{B} \Psi\) satisfies the group restricted isometry property (group-RIP) of order \(s\) with constant \(\delta_s \leq \delta\), provided that

\[
m \gtrsim_{\tau} \delta^{-2} \left[ s^2 \mu_{\mathcal{S}}^2 \log^2 D \left( \log^2 G + g^2 \log^2 \left( \frac{s}{\mu_{\mathcal{S}}} \right) \right) + \log \left( \frac{1}{\eta} \right) \right],
\]

where $\mu_{\mathcal{S}}$ and $\boldsymbol{\psi}_i$ is denote as above.
\end{mytheo}

\begin{proof}
To apply Corollary \ref{t4.1}, we first need an equivalent form to (\ref{ko2}), i.e. for $\forall \mathbf{x} \in \Sigma_{\mathcal{S}, s} \backslash\{\mathbf{0}\}$,
$$
\left|\frac{\|\mathbf{B} \boldsymbol{\Psi} \mathbf{x}\|_2^2}{\|\mathbf{x}\|_2^2}-1\right| \leq \delta .
$$

Define the set $\Omega$ as the $s$-group-sparse unit vectors:
$$
\Omega := \left\{ \mathbf{x} \in \mathbb{S}^{D-1} \mid \|\mathbf{x}\|_{\mathcal{S}, 0} \leq s \right\}.
$$

Using this, the group restricted isometry constant of $\mathbf{B}$ can be expressed as

\begin{equation}\label{te8}
\delta_s=\sup _{\mathbf{x} \in \Omega}\left|\|\mathbf{B} \boldsymbol{\Psi} \mathbf{x}\|_2^2-1\right| .
\end{equation}

Next, we reshape the aforementioned expression to align with the format stipulated in Corollary \ref{t4.1}, specifically by reformulating the equation to consider the supremum over a matrix set. To accomplish this, let us revisit the definition of the partial basis expansion matrices $\boldsymbol{\Psi}_l \in \mathbb{C}^{d \times d L}$ where $\boldsymbol{\Psi}=\left(\boldsymbol{\Psi}_1^{\top}, \ldots, \boldsymbol{\Psi}_L^{\top}\right)^{\top}$. 

To reformulate the previous expression into the form required by Theorem 4.1—specifically, to have the supremum taken over a set of matrices—we begin by recalling the definition of the partial basis expansion matrices \( \Psi_{l} \in \mathbb{C}^{d \times dL} \), where \( \Psi = \left( \Psi_{1}^\top, \ldots, \Psi_{L}^\top \right)^\top \).

Using Equation (\ref{koe1}), the $l$-th measurement vector \( y^{l} \in \mathbb{C}^{m} \) (from the overall vector \( y \in \mathbb{C}^{mL} \)) can be expressed as:

$$
\begin{aligned}
\mathbf{y}^l & =\boldsymbol{\Phi}_l \boldsymbol{\Psi}_l \mathbf{x}=\left(\begin{array}{c}
\left\langle\left(\boldsymbol{\Phi}_l\right)_1, \boldsymbol{\Psi}_l \mathbf{x}\right\rangle \\
\vdots \\
\left\langle\left(\boldsymbol{\Phi}_l\right)_m, \boldsymbol{\Psi}_l \mathbf{x}\right\rangle
\end{array}\right) ={\left(\begin{array}{ccc}
\left(\boldsymbol{\Psi}_l \mathbf{x}\right)^{\top} & & \\
& \ddots & \\
& \left(\boldsymbol{\Psi}_l \mathbf{x}\right)^{\top}
\end{array}\right)} \cdot {\left(\begin{array}{c}
\left(\boldsymbol{\Phi}_l\right)_1 \\
\vdots \\
\left(\boldsymbol{\Phi}_l\right)_m
\end{array}\right)},
\end{aligned}
$$
where $\left(\boldsymbol{\Phi}_l\right)_i \in \mathbb{C}^d$ denotes the $i$-th row of the matrix $\boldsymbol{\Phi}_l$, here, \( V_{l}(x) \) is a block-diagonal matrix with \( m \) blocks of \( (\Psi_{l} x)^\top \), and \( \xi^{l} \) is a concatenation of the measurement vectors \( (\Phi_{l})_{i} \) for \( i = 1, \ldots, m \).

Assuming that each block $\boldsymbol{\Phi}_l$ consists of independent copies of a $\varphi$-sub-Gaussian random variable, the combined vector $\boldsymbol{\xi} = \left( \left( \boldsymbol{\xi}^1 \right)^{\top}, \ldots, \left( \boldsymbol{\xi}^L \right)^{\top} \right)^{\top}$ is itself $\varphi$-sub-Gaussian. Define the linear operator $V: \mathbb{C}^{dL} \rightarrow \mathbb{C}^{mL \times mdL}$ by

$$
\mathbf{x} \mapsto V(\mathbf{x}) = \operatorname{diag}\left\{ V_l(\mathbf{x}) \right\}_{l=1}^L,
$$
where each $V_l(\mathbf{x})$ acts on $\mathbf{x}$. With this definition, we have the equality in distribution,
$
\mathbf{B} \boldsymbol{\Psi} \mathbf{x} \stackrel{\mathrm{d}}{=} V(\mathbf{x}) \boldsymbol{\xi},
$, where $\stackrel{\mathrm{d}}{=}$ denotes that both sides share the same distribution. Observing that

$$
\mathbb{E}\left\| \mathbf{B} \boldsymbol{\Psi} \mathbf{x} \right\|_2^2 = \mathbf{x}^* \boldsymbol{\Psi}^* \mathbb{E}\left[ \mathbf{B}^{\top} \mathbf{B} \right] \boldsymbol{\Psi} \mathbf{x} = m \left\| \mathbf{x} \right\|_2^2,
$$

we utilize the facts that the rows of each $\mathbf{B}_l$ are independent random $m$-vectors with independent entries and that $\boldsymbol{\Psi}$ is unitary. Therefore, applying Equation (\ref{te8}), the group restricted isometry property of the matrix $\frac{1}{\sqrt{m}} \mathbf{B} \boldsymbol{\Psi}$ can be formulated as

$$
\begin{aligned}
\delta_s\left(\frac{1}{\sqrt{m}} \mathbf{B} \boldsymbol{\Psi}\right) &= \sup_{\mathbf{x} \in \Omega} \left| \frac{1}{m} \left\| \mathbf{B} \boldsymbol{\Psi} \mathbf{x} \right\|_2^2 - \left\| \mathbf{x} \right\|_2^2 \right| \\
&= \frac{1}{m} \sup_{\mathbf{x} \in \Omega} \left| \left\| \mathbf{B} \boldsymbol{\Psi} \mathbf{x} \right\|_2^2 - m \left\| \mathbf{x} \right\|_2^2 \right| \\
&= \frac{1}{m} \sup_{\mathbf{x} \in \Omega} \left| \left\| \mathbf{B} \boldsymbol{\Psi} \mathbf{x} \right\|_2^2 - \mathbb{E} \left\| \mathbf{B} \boldsymbol{\Psi} \mathbf{x} \right\|_2^2 \right| \\
&\stackrel{\mathrm{d}}{=} \frac{1}{m} \sup_{\mathbf{x} \in \Omega} \left| \left\| V(\mathbf{x}) \boldsymbol{\xi} \right\|_2^2 - \mathbb{E} \left\| V(\mathbf{x}) \boldsymbol{\xi} \right\|_2^2 \right| \\
&= \frac{1}{m} \sup_{\boldsymbol{\Gamma} \in \mathcal{M}} \left| \left\| \boldsymbol{\Gamma} \boldsymbol{\xi} \right\|_2^2 - \mathbb{E} \left\| \boldsymbol{\Gamma} \boldsymbol{\xi} \right\|_2^2 \right|,
\end{aligned}
$$
where $\mathcal{M} = \{ V(\mathbf{x}) : \mathbf{x} \in \Omega \}$ is the set of operators derived from vectors in $\Omega$.

Then by the Corollary \ref{t4.1}, we need to estimate items $\gamma_2\left(\mathcal{M},\|\cdot\|_{l_2 \rightarrow l_2}\right), \gamma_{1}\left(\mathcal{M},\|\cdot\|_{l_2 \rightarrow l_{\infty}}\right)$, $M_F(\mathcal{M})$, and $M_{l_2 \rightarrow l_2}(\mathcal{M})$ separately. From Koep's results \cite{KOEP2022333}, we already know that $M_F(\mathcal{M})=\sqrt{m}$, $M_{l_2 \rightarrow l_{2}}(\mathcal{M})\leq \sqrt{s} \mu_{\mathcal{S}}$, and
$$
\begin{aligned}
\gamma_2\left(\mathcal{M},\|\cdot\|_{2 \rightarrow 2}\right) \lesssim  \sqrt{s} \mu_{\mathcal{S}} \sqrt{\log (D)} \log (s)\left(\sqrt{\log (G)}+\sqrt{g \log \left(s / \mu_{\mathcal{S}}\right)}\right).
\end{aligned}
$$
By virtue of $\Vert\cdot\Vert_{l_{2}\to l_{\infty}}\le \Vert\cdot\vert_{l_{2}\to l_{2}}$, we have
$$
\gamma_{1}\left(\mathcal{M},\|\cdot\|_{l_2 \rightarrow l_{\infty}}\right) \leq \gamma_{1}\left(\mathcal{M},\|\cdot\|_{l_2 \rightarrow l_2}\right).
$$
Then, we have by Dudley's bound
\begin{align}
	\gamma_{1}\left(\mathcal{M},\|\cdot\|_{l_2 \rightarrow l_2}\right)\lesssim \int_{0}^{\lambda}\log \mathcal{N}(\mathcal{M},  \Vert\cdot\Vert_{l_{2}\to l_{2}}, u)\,du+\int_{\lambda}^{\sqrt{s}\mu_{\mathcal{S}}}\log \mathcal{N}(\mathcal{M},  \Vert\cdot\Vert_{l_{2}\to l_{2}}, u)\,du,\nonumber
\end{align}
where $0\le \lambda\le \sqrt{s}\mu_{\mathcal{S}}$ is a parameter will be chosen later.

For the first part, we have (see (14) in \cite{KOEP2022333} for details)
\begin{align}
	\int_{0}^{\lambda}\log \mathcal{N}(\mathcal{M},  \Vert\cdot\Vert_{l_{2}\to l_{2}}, u)\,du\le \lambda s\log(eG/s)+2\lambda sg\log(5e/\lambda).\nonumber
\end{align}
As for the other part, we have (see Proposition 4.2 in \cite{KOEP2022333} for details)
\begin{align}
	\int_{\lambda}^{\sqrt{s}\mu_{\mathcal{S}}}\log \mathcal{N}(\mathcal{M},  \Vert\cdot\Vert_{l_{2}\to l_{2}}, u)\,du\le \sqrt{s}\mu_{\mathcal{S}}^{2}\log D(\frac{\sqrt{s}}{\lambda}-\frac{1}{\mu_{\mathcal{S}}})(\log G+g\log(s/\lambda)).\nonumber
\end{align}
Taking $\lambda=\mu_{\mathcal{S}}$, we have
\begin{align}
	\gamma_{1}\left(\mathcal{M},\|\cdot\|_{l_2 \rightarrow l_2}\right)\lesssim s\mu_{\mathcal{S}}\log D\big(\log G+g\log(s/\mu_{\mathcal{S}})   \big).\nonumber
\end{align}

Then by applying the Corollary \ref{t4.1} with the upper bound of all the quantity, the proof is completed.
\end{proof}

\textbf{Acknowledgment:} The authors are grateful to Hanchao Wang and Yu Xia for their fruitful discussions. 
Chen was partly supported by the National Natural Science Foundation of China  (No. 12071257).


\end{document}